\documentclass[11pt]{amsart} \usepackage[all]{xy} \usepackage{enumitem} \usepackage{mathtools} \usepackage{mathrsfs} \usepackage{amssymb} \usepackage{amsmath,amsfonts,amsthm} \usepackage{graphicx} \usepackage{upgreek} \usepackage[toc,page]{appendix} \usepackage{bbm} \usepackage{esint} \usepackage{scalerel} \usepackage{stackengine,wasysym} \usepackage{verbatim}  \usepackage{hyperref} \usepackage[dvipsnames]{xcolor} \hypersetup{colorlinks=true, linktocpage=true,citecolor=blue, linkcolor=blue, urlcolor=blue} \usepackage{geometry} 
\newtheorem{teo}{Theorem}[section]
\newtheorem{defi}[teo]{Definition}
\newtheorem{prop}[teo]{Proposition}

\newtheorem{lemma}[teo]{Lemma}

\newtheorem{rmk}[teo]{Remark}
\newcommand{\R}{\mathbb{R}}

\newcommand{\Ric}{\text{Ric}}

\newcommand{\diam}{\text{diam}}
\usepackage[english]{babel}
\usepackage{ragged2e}
\usepackage{amsmath}
\usepackage{mathtools}
\usepackage{xcolor}
\usepackage{amsfonts}
\usepackage{amsthm}

\title{Cheng's Theorem in the Spectral Setting}
\author{Carolina Ruviaro}
\address{IMPA - Instituto de Matemática Pura e Aplicada, Rio de Janeiro, RJ, Brasil.} \email{carolinaruviaro@gmail.com} 

\usepackage{setspace}
\begin{document}
\raggedbottom

\begin{abstract}
We prove a maximal diameter rigidity theorem for closed $3$-dimensional Riemannian manifolds under a spectral lower bound for the Ricci curvature. We show that if $0<\gamma<2$, $K>0$ and $\lambda_1(-\gamma\Delta+\Ric)\ge 2K$, then the optimal bound $\text{diam}(M,g)\leq\frac{\pi}{\sqrt{K}}$ is attained if and only if $(M,g)$ is isometric to the round $3$-sphere of constant sectional curvature $K$.
\end{abstract}

\maketitle \setcounter{tocdepth}{1}

\section{Introduction}

The classical Bonnet-Myers theorem asserts that if a complete Riemannian manifold $(M^n,g)$ satisfies $\Ric_g\geq (n-1)Kg$, then $\diam (M,g)\leq \frac{\pi}{\sqrt{K}}.$ The equality case was characterized by Cheng \cite{Cheng1975}: if $\Ric_g\geq (n-1) Kg$ and $\diam(M,g)=\frac{\pi}{\sqrt{K}},$ then $(M^n,g)$ is isometric to the round $n$-sphere of sectional curvature $K$. In recent years several classical results that traditionally rely on pointwise lower bounds for the Ricci curvature have been revisited under the weaker assumption of a spectral Ricci curvature lower bound. These include spectral analogues of splitting theorems, volume comparison results and Bonnet-Myers type estimates \cite{AntonelliPozzettaXu2024,AntonelliXuBG,HongWang}. 

In this paper, given a Riemannian manifold $(M^n,g)$, a spectral lower bound for the Ricci curvature is formulated in terms of the operator $-\gamma\Delta_g+\rho_g$, where $\gamma>0$ is a constant and \[\rho_g(x) \coloneqq \min_{\substack{v\in T_xM\\ \vert v\vert_g=1}} \Ric_g(v,v),\] that is, $\rho_g(x)$ provides the smallest eigenvalue of the Ricci tensor at $x \in M$. For a constant $k\in \R$, we say that $(M,g)$ satisfies the spectral condition $\lambda_1(-\gamma \Delta_g + \rho_g) \geq k$ if there exists a positive function $u \in C^{2,\alpha}(M)$ with $\alpha \in (0,1)$ such that $-\gamma \Delta_g u + \rho_g u \geq k u$. Equivalently, if for every smooth function $\phi$ with compact support it holds the inequality \[\int_M (\gamma \vert \nabla\phi\vert^2 + \rho_g \phi^2) \ dV_g\ \geq k \int_M \phi^2 \ dV_g. \] Observe that when $\gamma=0$ the spectral condition reduces to a pointwise curvature hypothesis and for that reason throughout this paper we assume $\gamma\neq 0$.

Turning to dimension three, when $0<\gamma<2$ the spectral lower bound $\lambda_1(-\gamma\Delta_g+\rho_g)\geq 2$ implies the sharp diameter estimate $\diam (M,g)\leq \pi$. This was proved by Antonelli-Xu \cite{AntonelliXuBG} and recently Hong-Wang \cite{HongWang} gave a different proof that generalizes to all dimensions $n\geq 3$, providing an upper bound that matches the one obtained by Shen-Ye \cite{ShenYe1997}. Thus in dimension three the spectral condition recovers the same diameter limit as the classical assumption $\Ric_g\geq 2g$ given by the Bonnet-Myers theorem. This naturally raises the question of whether Cheng's maximal diameter rigidity theorem admits a spectral version. The main result of this paper gives an affirmative answer to it.

\begin{teo} \label{Spectral Cheng}
Let $(M^3,g)$ be a closed connected Riemannian manifold, $K>0$ and $0<\gamma <2$. If $\lambda_1(-\gamma\Delta_g+\rho_g)\geq 2K,$ then $\diam (M,g)\leq\frac{\pi}{\sqrt K}$ and equality holds if and only if $(M^3,g)$ is isometric to the round $3$-sphere of constant sectional curvature $K$.
\end{teo}

 We emphasize that our approach leading to the sharp diameter bound differs from that of \cite{HongWang}. Our proof is based on the weighted isoperimetric profile and it was inspired by \cite{AntonelliXuBG}. The spectral lower bound yields, through a second variation argument, a differential inequality for this profile which provides an upper bound for a quantity that involves the inverse of the weighted isoperimetric profile. A coarea argument for geodesic balls shows that this quantity controls the Riemannian diameter, which gives another proof for the estimate in dimension three different from the one contained in \cite{AntonelliXuBG}. In the equality case the rigidity of each step forces the function guaranteed by the spectral hypothesis to be constant, reducing the spectral condition to the classical lower Ricci curvature bound and so Cheng's theorem applies.

We also note that the range $0<\gamma<2$ in Theorem \ref{Spectral Cheng} is sharp. Indeed, at the critical value $\gamma=2$ we construct a family of rotationally symmetric metrics $g_\varepsilon$ on $\mathbb S^3$, which are not the round metric, satisfying both hypotheses \[\lambda_1 (-2\Delta_{g_\varepsilon}+\rho_{g_\varepsilon})=2 \qquad\text{and} \qquad \diam (\mathbb S^3,g_\varepsilon)=\pi,\]
showing that the diameter rigidity fails.

Moreover we discuss the higher dimensional case of Theorem \ref{Spectral Cheng}. The sharp spectral Bonnet-Myers estimate of Hong-Wang \cite{HongWang} shows that, when $n>3$ and $0<\gamma<\frac{4}{n-1}$, the natural lower bound \[\lambda_1(-\gamma\Delta_g+\rho_g)\geq n-1\] no longer yields the classical diameter bound $\pi$. In fact, using their sharpness examples, we show that for $n>3$ and such $\gamma$ there is no constant $C=C(n,\gamma)$ such that the condition $\lambda_1(-\gamma\Delta_g+\rho_g)\geq C$ simultaneously implies  $\diam (M,g)\leq\pi$ and guarantees that the equality case  $\diam (M,g)=\pi$ characterizes the unit sphere.

The paper is organized as follows: in Section \ref{section preliminary} we present some preliminary calculations, in Section \ref{section proof} we prove Theorem \ref{Spectral Cheng}, in Section \ref{section range gamma} we study the sharpness of the range $0<\gamma<2$ and construct the counterexamples at $\gamma=2$. Finally in Section \ref{section case n>3} we discuss the higher dimensional case and explain why a spectral analogue of Cheng theorem cannot hold when $n=2$ or $n> 3$.

\section*{Acknowledgements}
The author would like to thank her advisor, Lucas Ambrozio, for his helpful suggestions. This study was financed in part by the Coordenação de Aperfeiçoamento de Pessoal de Nível Superior - Brasil (CAPES) - Finance Code 001 and FAPERJ - Fundação Carlos Chagas Filho de Amparo à Pesquisa do Estado do Rio de Janeiro, Processo SEI-260003/001431/2026.

\section{Preliminary Calculations}\label{section preliminary}

In this section we use the notion of a function satisfying a differential inequality in the \textit{viscosity sense}. The definition given here is for the specific equation that we will use in this paper, but this notion holds more generally.

\begin{defi}
    A lower semicontinuous function $I\colon [0,V] \longrightarrow [0,+\infty)$ satisfies, for some $k>0$, the inequality
    \begin{equation}\label{definition viscosity}
        I''I +\dfrac{1}{2}(I')^2 \leq -2k
    \end{equation}
    in the viscosity sense at a point $v_0 \in (0,V)$ if for every smooth function $\varphi$ defined in some neighborhood $\mathcal V$ of $v_0$ such that $\varphi(v_0) = I(v_0)$ and $\varphi(x) \leq I(x)$ in $\mathcal V$, we have $\varphi(v_0) \varphi''(v_0) +\dfrac{1}{2}(\varphi'(v_0))^2 \leq -2k$ in the classical sense. The inequality \eqref{definition viscosity} is satisfied in the viscosity sense on the interval $(0,V)$ if the viscosity condition holds for every point in $(0,V)$.
\end{defi}

Based on this idea, we prove three lemmas that will be useful in Section \ref{section proof}.

\begin{lemma}\label{lemma integrability}
Let $V>0$ and $I \colon [0,V]\longrightarrow[0,+\infty)$ be a continuous function which is positive on $(0,V)$ and satisfies $I(0)=I(V)=0.$ Assume that there exists a constant $c>0$ such that \[c\min\{v,V-v\}^{2/3} \leq I(v)\] for $v$ sufficiently close to $0$ and $V$. Then the integral of $\frac{1}{I(v)}$ over $(0,V)$ is finite.
\end{lemma}

\begin{proof}
Near $v=0$, the lower bound gives \[\frac1{I(v)} \leq \frac{1}{c} v^{-2/3} \] and $v^{-2/3}$ is integrable near $0$. Similarly, $\frac{1}{I(v)}$ is integrable near $V$. Since $I$ is positive and continuous on every compact subinterval of $(0,V)$, the conclusion follows.
\end{proof}

\begin{lemma}\label{lemma bound on L}
Consider $V>0$, $k>0$ and let $I\colon [0,V]\longrightarrow[0,+\infty)$ be a continuous function which is positive on $(0,V)$ and satisfies $I(0)=I(V)=0.$ Assume that \[L\coloneqq\int_0^V\frac{dv}{I(v)}<+\infty\] and that \[ I''I \leq -\frac{1}{2}(I')^2-2k\] in the viscosity sense on $(0,V)$. Then $L\leq\frac{\pi}{\sqrt{k}}.$
\end{lemma}

\begin{proof}
    The viscosity inequality means that if a smooth function $\varphi$ touches $I$ from below at $v_0\in(0,V)$ then $I(v_0)\varphi''(v_0) \leq - \frac{1}{2}\bigl(\varphi'(v_0)\bigr)^2-2k.$ Since $I(v_0)>0$ this implies in particular that $\varphi''(v_0)\leq0.$ Hence $I$ is concave on $(0,V)$. In particular, it is locally Lipschitz and hence locally absolutely continuous. Moreover, for a concave function the right derivative $p(v)\coloneqq I'_+(v)$ exists at every $v\in(0,V)$ and is nonincreasing. In particular, $p$ is differentiable almost everywhere.
    
    We claim that at every point $v_0$ where $p$ is differentiable holds
\begin{equation}\label{inequality I and p}
    I(v_0)p'(v_0) \leq -\frac{1}{2}p(v_0)^2-2k.
\end{equation}
Indeed, locally absolute continuity of $I$ and differentiability of $p$ at $v_0$ provides for small $h$ the equalities
\begin{align*}
    I(v_0+h)-I(v_0) &= \int_0^h p(v_0+t)\ dt = \int_0^h p(v_0)+p'(v_0)t + o(\vert t \vert)\ dt\\
    &=p(v_0) h +\dfrac{1}{2}p'(v_0)h^2 + o(h^2).
\end{align*}
Thus for every $\varepsilon>0$, the function \[\varphi_\varepsilon(v) = I(v_0) + p(v_0)(v-v_0) +\frac{1}{2}\bigl(p'(v_0)-2\varepsilon\bigr)(v-v_0)^2\] in a small neighborhood of $v_0$ touches $I$ from below. Then the viscosity inequality gives \[I(v_0)\bigl(p'(v_0)-2\varepsilon\bigr) \leq -\frac{1}{2}p(v_0)^2-2k.\] Letting $\varepsilon\rightarrow 0$, we conclude that \eqref{inequality I and p} holds.

Now define \[r(v)\coloneqq \int_0^v\frac{ds}{I(s)}.\] Since $I>0$ on $(0,V)$, $r$ is strictly increasing and maps $[0,V]$ onto $[0,L]$. If we let $v=v(r)$ be its inverse, we have $v'(r) = I(v(r))>0$ on the interior $(0,L)$. Defining \[q(r)\coloneqq \frac{1}{2} p(v(r)),\] as $p$ is nonincreasing and $v$ is increasing, $q$ is also nonincreasing. Hence at almost every $r$, the chain rule and \eqref{inequality I and p} give \[q'(r) = \frac{1}{2} p'(v(r))v'(r) = \frac{1}{2} I(v(r))p'(v(r)) \leq -\frac14p(v(r))^2-k = -q(r)^2-k.\] Hence for almost every $r \in (0,L)$ we have: \[q'(r)+q(r)^2+k\leq 0.\] Set \[ \theta(r) \coloneqq  \arctan\left(\frac{q(r)}{\sqrt{k}}\right). \] Since $q$ is nonincreasing, so is $\theta$. Moreover, at almost every point, \[ \theta'(r) = \frac{\sqrt{k}}{k+q(r)^2} q'(r) \leq -\frac{\sqrt{k}}{k+q(r)^2} \bigl(q(r)^2+k\bigr) = -\sqrt{k}.\] Now let $0<a<b<L$. Because $\theta$ is nonincreasing, 
\begin{align}\label{difference theta}
    \theta(b)-\theta(a) \leq \int_a^b\theta'(r)\ dr \leq -\sqrt{k}(b-a).
\end{align}
Thus $\sqrt{k}(b-a) \leq \theta(a)-\theta(b).$ But $-\frac{\pi}{2} \leq \theta(r) \leq \frac{\pi}{2}$ and so $\theta(a)-\theta(b)\leq\pi.$ Therefore \[ b-a\leq\frac{\pi}{\sqrt{k}}\] and letting $a\rightarrow 0$ and $b\rightarrow L$, we conclude that $L\leq\frac{\pi}{\sqrt{k}}.$
\end{proof}

\begin{lemma}\label{lemma bound on L-equality}
Under the hypotheses of Lemma \ref{lemma bound on L}, assume that $L=\frac{\pi}{\sqrt{k}}.$ Then there exists a constant $C>0$ such that \[I(v(r)) = C^2\sin^2(\sqrt{k}\ r),\] where $v=v(r)$ is the inverse of $r(v)\coloneqq \int_0^v I(s)^{-1}\ ds$. In particular, $I$ is smooth on $(0,V)$ and satisfies in the classical sense the equation \[I''I = -\frac{1}{2}(I')^2-2k.\]
\end{lemma}

\begin{proof}
Assume that $L=\frac{\pi}{\sqrt{k}}.$ We keep the notation from the proof of Lemma \ref{lemma bound on L}: \[p(v)\coloneqq I'_+(v), \qquad q(r)\coloneqq \frac{1}{2} p(v(r)) \qquad \text{and} \qquad \theta(r)\coloneqq \arctan\left(\frac{q(r)}{\sqrt{k}}\right).\]
It is already known by inequality \eqref{difference theta} that for every $0<a<b<L$ \[\theta(b)-\theta(a) \leq -\sqrt{k}(b-a).\] Since $\theta$ is nonincreasing, the extreme limits $\theta(0+)$ and $\theta(L-)$ exist. Letting $a\rightarrow 0$ and $b\rightarrow L$, we obtain \[\theta(L-)-\theta(0+) \leq -\sqrt{k}L = -\pi.\] On the other hand, $-\frac{\pi}{2} \leq \theta(r) \leq \frac{\pi}{2}$ and therefore $\theta(L-)-\theta(0+)\geq-\pi$ which gives the equality $\theta(L-)-\theta(0+)=-\pi.$ So since $\theta$ is nonincreasing, $\theta(0+)=\frac{\pi}{2}$ and $\theta(L-)=-\frac{\pi}{2}.$ 

We claim that equality holds in \eqref{difference theta} for every $0<a<b<L$. Indeed, if the inequality were strict for some $a<b$, there would exist $\delta>0$ such that $\theta(b)-\theta(a) = -\sqrt{k}(b-a) - \delta$. Choose $\varepsilon>0$ such that $0<\varepsilon<a<b<L-\varepsilon.$ Applying \eqref{difference theta} on the intervals $(\varepsilon,a)$ and $(b,L-\varepsilon)$ and adding the two inequalities with the difference $\theta(b) - \theta(a)$ we would obtain \[\theta(L-\varepsilon)-\theta(\varepsilon) \leq -\sqrt{k}(L-2\varepsilon) - \delta.\] Letting $\varepsilon\rightarrow 0$, we would get \[\theta(L-)-\theta(0+) \leq -\sqrt{k}L -\delta= -\pi-\delta < -\pi\] which is a contradiction. Therefore \[\theta(b)-\theta(a) = -\sqrt{k}(b-a)\] for every $0<a<b<L$ and hence for all $r \in (0,L)$, \[\theta(r) = \frac{\pi}{2}-\sqrt{k}\ r.\] By definition of $\theta(r)$ we obtain \[q(r) = \sqrt{k}\cot(\sqrt{k}\ r).\] Now define \[F(r)\coloneqq \sqrt{I(v(r))}.\]

We claim that $F$ is locally absolutely continuous on $(0,L)$. Indeed, since $I$ is concave, it is locally Lipschitz and since $v$ is Lipschitz, $F$ is also locally Lipschitz. In particular, it is locally absolutely continuous on $(0,L)$. Hence at almost every $r\in(0,L)$ the chain rule gives
\[F'(r) = \frac{1}{2\sqrt{I(v(r))}}
I'(v(r))v'(r) = \frac{1}{2\sqrt{I(v(r))}}
I'(v(r))I(v(r)) = \frac{1}{2} I'(v(r))F(r). \]
At every point where $I'$ exists we know that $I'(v(r))=p(v(r))$ and hence
\[\frac{F'(r)}{F(r)} = \frac{1}{2}p(v(r)) = q(r) = \sqrt{k}\cot(\sqrt{k}\ r)\]
for almost every $r\in(0,L)$. But \[\frac{d}{dr} \log\sin(\sqrt{k}\ r) = \sqrt{k}\cot(\sqrt{k}\ r). \] Therefore \[ \frac{d}{dr} \left[\log F(r) - \log\sin(\sqrt{k}\ r) \right] = 0 \] almost everywhere on $(0,L)$. Since the function in brackets is locally absolutely continuous, it is
constant. Thus there exists $C>0$ such that \[F(r)= C\sin(\sqrt{k}\ r). \] Recalling that $F(r)^2=I(v(r)),$ we obtain \[I(v(r))= C ^2\sin^2(\sqrt{k}\,r).\] Now \[v'(r) = I(v(r)) = C ^2\sin^2(\sqrt{k}\ r)\] which is smooth and strictly positive for $r\in(0,L)$. Hence $v$ is smooth on $(0,L)$ with nonvanishing derivative. Therefore its inverse $r=r(v)$ is smooth on $(0,V)$. Since \[I(v) = C^2\sin^2\bigl(\sqrt{k}\,r(v)\bigr),\] it follows that $I\in C^\infty((0,V)).$ We may now differentiate classically. Since $q(r) = \sqrt{k}\cot(\sqrt{k}\ r)$, we have $q'+q^2+k=0.$ On the other hand, $q(r) = \frac{1}{2}I'(v(r))$ and so \[ q'(r) = \frac{1}{2}I''(v(r))v'(r) = \frac{1}{2}I(v(r))I''(v(r)).\] Therefore \[\frac{1}{2}I''I+ \frac{1}{4}(I')^2 + k = 0\] and hence $I''I = -\frac{1}{2}(I')^2-2k$.
\end{proof}

\section{Proof of Theorem \ref{Spectral Cheng}} \label{section proof}

We now apply the computational lemmas of Section \ref{section preliminary} to the spectral Ricci setting. By scaling we may assume without loss of generality that $K=1$. Since $\lambda_1(-\gamma \Delta_g + \rho_g) \geq 2$, take $u$ a positive $C^{2,\alpha}(M)$ function which satisfies $-\gamma\Delta_gu+\rho_gu \geq 2u.$ By Remark \ref{remark regularity} below we can suppose that $u$ is actually a smooth function. Note that the multiplication of $u$ by a positive constant does not affect the above inequality, so we can always assume that $\min_Mu=1.$ Define the weighted measure and weighted perimeter by \[d\mu\coloneqq u^\gamma\ dV_g \qquad  \text{and} \qquad P_\mu(E)\coloneqq\int_{\partial^*E}u^\gamma\ d\mathcal H_g^2\] where $E$ is a set of finite perimeter and $\partial^*E$ is its reduced boundary. We also fix $V\coloneqq\mu(M)=\int_Mu^\gamma\ dV_g.$ Consider the weighted isoperimetric profile \[I(v) \coloneqq \inf\left\{P_\mu(E)\colon E\subset\subset M\text{ has finite perimeter and } \mu(E)=v \right\}\] for $0\leq v\leq V$. Since $M$ is compact and $u$ is positive, isoperimetric minimizers exist for every $v\in(0,V)$ (see, for instance, Theorem 2.1 in \cite{Castro2022}). As $n=3$, their boundaries are smooth by standard facts of Geometric Measure Theory. Furthermore, $I$ is positive on $(0,V)$ and $I(0)=I(V)=0$. By Lemma 2 in \cite{AntonelliXuBG} the profile satisfies
\begin{equation}\label{viscosity}
    I''I \leq -\frac{(I')^2}{2}-2
\end{equation}
in the viscosity sense on $(0,V)$. We emphasize that the computation leading to \eqref{viscosity} is the first and second variation computation in \cite[equations (34)-(44)]{AntonelliXuBG} and we shall return to those computations for the analysis of the equality case. 

Since $M$ is a compact manifold, there exists a constant $C>0$ such that \[ C\min\{v,V-v\}^{2/3} \leq I(v)\] near the endpoints. Indeed, if $0<v\leq \frac{V}{2}$, for all sets $E\subset \subset M$ with $\mu(E) =v$ and finite perimeter Lemma 3.1 in \cite{Castro2022} gives a constant $C>0$ such that $P_\mu(E) \geq C\ v^{\frac{2}{3}}$. Taking the infimum, it follows that $I(v) \geq C \ v^{\frac{2}{3}}$. Since the isoperimetric profile satisfy the symmetric property $I(s) = I(V-s)$, we have that for $v\geq \frac{V}{2}$, $I(v) = I(V-v) \geq C \ (V-v)^{\frac{2}{3}}$ and thus we have the desired lower bound.
 
Once $I$ is continuous (see, for instance, Proposition 5.6 in \cite{ChodoshLiMinterStryker2026}), $I(0)=I(V)=0$ and it is positive on $(0,V)$, Lemmas \ref{lemma integrability} and \ref{lemma bound on L} therefore apply and give
\begin{equation}\label{Ru(M)}
    L_u(M) \coloneqq \int_0^V\frac{dv}{I(v)} \leq \pi.
\end{equation}
We next compare the quantity $L_u(M)$ with the Riemannian diameter. We observe that a similar strategy was used in \cite{Gallot1988} for other purposes. 

Fix $p\in M$ and let $R_p\coloneqq\max_{x\in M}d_g(p,x).$ Define \[m_p(r) \coloneqq \mu(B_r(p)) = \int_{B_r(p)}u^\gamma\ dV_g, \qquad 0\leq r\leq R_p,\] where $B_r(p) = \{x  \in M \colon d_g(p,x)<r\}$. The distance function $d_p=d_g(p,\cdot)$ is Lipschitz and $\vert \nabla d_p\vert=1$ almost everywhere on $M$. By the coarea formula, \[ m_p'(r) = P_\mu(B_r(p))\] for almost every $r\in(0,R_p)$. In particular, by the definition of $I$,
\begin{equation}\label{m'}
    m_p'(r) \geq I(m_p(r))
\end{equation}
for almost every $r\in(0,R_p)$. Since $m_p(0)=0$, $m_p(R_p)=V$ and $m_p$ is absolutely continuous, the change of variables formula gives
\begin{align}\label{Ru(M) inequality}
    L_u(M) = \int_0^V\frac{dv}{I(v)} = \int_0^{R_p} \frac{m_p'(r)}{I(m_p(r))}\ dr \geq \int_0^{R_p}1\ dr = R_p.  
\end{align}
Taking the supremum over $p$ yields
\begin{equation}\label{diam and Ru(M)}
    \diam(M,g) \leq L_u(M).
\end{equation}
Combining \eqref{diam and Ru(M)} with
\eqref{Ru(M)}, we obtain
\begin{equation}\label{diameter chain}
    \diam(M,g) \leq L_u(M) \leq \pi.
\end{equation}
This proves the diameter estimate. We now consider the equality case, that is, assume that $\diam(M,g)=\pi.$ The chain of inequalities \eqref{diameter chain} gives \[ \pi = \diam(M,g) \leq L_u(M) \leq \pi\] and therefore $L_u(M)=\pi$. By Lemma \ref{lemma bound on L-equality}, the profile $I$ is smooth on $(0,V)$ and satisfies
\begin{equation}\label{viscosity equality}
    I''I = -\frac{1}{2}(I')^2-2
\end{equation}
pointwise in the classical sense. Choose a point $p\in M$ such that $R_p=\pi.$ We then obtain by equality in \eqref{Ru(M) inequality} that \[0 = L_u(M)-R_p = \int_0^\pi \left(\frac{m_p'(r)}{I(m_p(r))}-1 \right)\ dr.\] Since the integrand is nonnegative by \eqref{m'}, we have \[m_p'(r)=I(m_p(r))\] for almost every $r\in(0,\pi)$. Therefore for almost every $r\in(0,\pi)$,
\begin{equation}\label{balls minimizer}
    P_\mu(B_r(p)) = I(\mu(B_r(p)))
\end{equation}
and $B_r(p)$ is a weighted isoperimetric region for almost every $r$. Since we are in dimension three, by standard facts of Geometric Measure Theory this in particular implies that these balls have smooth boundary. 

We now return to the second variation computation in Lemma 2 of \cite{AntonelliXuBG} and keep track of two terms which were there discarded. Fix a radius $r$ for which \eqref{balls minimizer} holds and put \[E\coloneqq B_r(p), \qquad \Sigma\coloneqq\partial E \qquad \text{and} \qquad v_0\coloneqq\mu(E).\] Consider a smooth family $\{E_t\}_{t \in (-\varepsilon,\varepsilon)}$ with $E_0 = E$ and consider a normal variation with speed $\varphi = u^{-\gamma}$, that is, take as the variational vector field along $\partial E_t$ at $t=0$ the vector $\varphi \nu$, where $\nu$ is the outer unit normal of $\partial E_t$, and normal acceleration $(\varphi \varphi_\nu)\nu$ at $t=0$, where $\varphi_\nu \coloneqq \langle \nabla \varphi, \nu\rangle$. Put $A(t) = \int_{\partial E_t}u^{\gamma}$ and $V(t) = \int_{ E_t}u^{\gamma}$. The second variation at $t=0$ is given by
\[\begin{aligned}
\frac{d^2 A}{dt^2}(0)
={}& \int_{\Sigma} -\Delta_{\Sigma}u^{-\gamma} -\Ric(\nu,\nu)u^{-\gamma} -\Vert \Pi\Vert^2u^{-\gamma} -\gamma u^{-2-\gamma}u_\nu^2 \\
&\quad +\gamma u^{-\gamma-1} \left( \Delta u-\Delta_{\Sigma}u-Hu_\nu \right) -\gamma u^{-1}
\left\langle \nabla_{\Sigma}u, \nabla_{\Sigma}\left(u^{-\gamma}\right) \right\rangle\\
&\quad+Hu^{-\gamma}\left(H+\gamma u^{-1}u_\nu\right) \\
={}& \int_{\Sigma} -\Ric(\nu,\nu)u^{-\gamma}
+\gamma u^{-\gamma-1}\Delta u +\left[\gamma(-\gamma-1)+\gamma^2\right] \left\vert\nabla_{\Sigma}u\right\vert^2u^{-\gamma-2}\\
&\quad -\Vert\Pi\Vert^2u^{-\gamma} -\gamma u^{-\gamma-2}u_\nu^2 -\gamma H u^{-\gamma-1}u_\nu
+H u^{-\gamma}\left(H+\gamma u^{-1}u_\nu\right)
\end{aligned}
\]
where $H$ is the mean curvature of $\Sigma=\partial E$, $u_\nu \coloneqq\langle \nabla u, \nu\rangle$ and $\Pi$ denotes the second fundamental form. We can estimate the second variation as:
\[\begin{aligned}
\frac{d^2 A}{dt^2}(0) \leq{}& \int_{\Sigma} -2u^{-\gamma} -\gamma\vert \nabla_{\Sigma}u\vert^2u^{-\gamma-2}- \dfrac{u^{-\gamma}}{2}H^2 - \gamma u^{-\gamma-2}u^2_\nu\\
&- \gamma H u^{-\gamma-1}u_\nu + Hu^{-\gamma}(H+\gamma u^{-1}u_\nu).\\
={}& \int_\Sigma -2u^{-\gamma}-\gamma \vert \nabla_\Sigma u\vert^2 u^{-\gamma-2}+\dfrac{u^{-\gamma}}{2}H^2 -\gamma u^{-\gamma-2}u^2_\nu.
\end{aligned}\]
Because $t \mapsto V(t)$ is strictly monotone in a neighborhood of $0$, we can interpret $t\mapsto A(t)$ as a function in $V$ which is defined in a neighborhood of $v_0$. In this case we have $A(v_0)= \int_\Sigma u^\gamma=I(v_0)$ and $I(v)\leq A(v)$ for $v$ close to $v_0$. Since $I$ is smooth we have:
\begin{equation}\label{properties A(v0)}
    A'(v_0)=I'(v_0) \qquad \text{and} \qquad A''(v_0)\geq I''(v_0).
\end{equation}
Since $E$ minimizes the weighted perimeter under a fixed weighted volume constraint, the first variation implies that $H+\gamma u^{-1}u_\nu$ is constant on $\Sigma$. The chain rule then identifies this constant with $A'(v_0)$. Thus, setting $X\coloneqq A'(v_0)$ and $Y:=u^{-1}u_\nu$, we have $H=X-\gamma Y$. With this, the second variation estimate can be rewritten as follows:
\begin{align}
\frac{d^2 A}{dt^2}(0)&\leq \int_{\Sigma} -2u^{-\gamma} -\gamma\vert \nabla_{\Sigma}u\vert^2u^{-\gamma-2}+ \dfrac{u^{-\gamma}}{2}H^2 - \gamma u^{-\gamma}Y^2 \nonumber \\ 
&=\int_{\Sigma}-2u^{-\gamma} - \gamma\vert \nabla_{\Sigma}u\vert^2u^{-\gamma-2} +u^{-\gamma}\left[\gamma\left(\dfrac{\gamma}{2}-1 \right)Y^2 + \dfrac{X^2}{2} - \gamma XY\right] \label{second variation area estimate}.
\end{align}
Using the chain rule to evaluate $A''(v_0)$, we see that for $Q \coloneqq \int_{\Sigma} 1 = \vert \Sigma \vert = \frac{dV}{dt}(0)$,
\[A''(v_0)= Q^{-2}\frac{d^2A}{dt^2}(0) -Q^{-3}\frac{dA}{dt}(0) \ \frac{d^2V}{dt^2}(0)\] and so \[Q^2 A''(v_0) = \frac{d^2A}{dt^2}(0) - X \ \frac{d^2V}{dt^2}(0).\] Notice that the second term is actually:
\begin{align*}
    -X \ \dfrac{d^2V}{dt^2}(0) &= -X \int_\Sigma \left[(H+\gamma Y)u^{-\gamma} + (u^{-\gamma})_\nu\right] = -X \int_\Sigma Hu^{-\gamma}\\
    &= -X\int_\Sigma(X-\gamma Y)u^{-\gamma} = -X^2\int_\Sigma u^{-\gamma} +\gamma X \int_\Sigma Yu^{-\gamma}
\end{align*}
and combining with \eqref{second variation area estimate} we obtain:
\begin{align}
Q^2A''(v_0) \leq{}& -\left(\frac{A'(v_0)^2}{2}+2
\right) \int_\Sigma u^{-\gamma} \label{Ru(M) inequalityefined-AX}-\gamma \int_\Sigma u^{-\gamma} \vert\nabla_\Sigma\log u\vert^2 \nonumber\\
&- \gamma\left(1-\frac{\gamma}{2}\right)\int_\Sigma u^{-\gamma} \left(\frac{u_\nu}{u}\right)^2. \nonumber
\end{align}
By Hölder inequality,
\[Q^2 = \left(\int_\Sigma 1\right)^2 \leq \left(\int_\Sigma u^\gamma\right) \left(\int_\Sigma u^{-\gamma}\right) = A(v_0)\int_\Sigma u^{-\gamma}\]
and so \[-\dfrac{A(v_0)}{Q^2}  \leq -\left( \int_\Sigma u^{-\gamma}\right)^{-1}.\]
We then deduce that
\begin{align}
A(v_0)A''(v_0) \leq{}& -\frac{1}{2}A'(v_0)^2-2 \label{Ru(M) inequalityefined-AX-final} -\frac{A(v_0)}{Q^2}\gamma \int_\Sigma u^{-\gamma} \vert\nabla_\Sigma\log u\vert^2 \nonumber\\
&-\frac{A(v_0)}{Q^2} \gamma\left(1-\frac{\gamma}{2}\right) \int_\Sigma u^{-\gamma} \left(\frac{u_\nu}{u}\right)^2. \nonumber
\end{align}
On the other hand, by \eqref{viscosity equality} and
\eqref{properties A(v0)} we have
\[-\frac{1}{2}A'(v_0)^2-2 = -\frac{1}{2}I'(v_0)^2-2 = I''(v_0)I(v_0) \leq A''(v_0)A(v_0)\] and this yields \[ 0 \leq - \frac{I(v_0)}{Q^2} \gamma \int_\Sigma u^{-\gamma} \vert\nabla_\Sigma\log u\vert^2 - \frac{I(v_0)}{Q^2} \gamma \left(1-\frac{\gamma}{2}\right) \int_\Sigma u^{-\gamma} \left(\frac{u_\nu}{u}\right)^2 \leq 0.\]
Hence both integrals vanish and we have $\nabla_\Sigma u=0$ and $u_\nu=0$ on $\Sigma$. Thus for almost every $r\in(0,\pi)$
\begin{equation}\label{gradient zero}
    \nabla u=0 \qquad\text{on }\partial B_r(p).
\end{equation}
Finally we apply the coarea formula once more to the distance function $d_p$ together with \eqref{gradient zero} to obtain
\[ \int_M\vert \nabla u\vert\ dV_g = \int_0^\pi \left(\int_{\{d_p=r\}} \vert\nabla u\vert \ d\mathcal H_g^2\right)dr=0.\] Hence $\nabla u=0$  almost everywhere on $M$ and since $u$ is smooth, $\nabla u\equiv0$ on $M$. As $M$ is connected, $u$ is constant. The hypothesis $\gamma\Delta_g u \leq \rho_g u - 2u$ then reduces to $\rho_g \geq 2$ and hence $\Ric_g\geq2g.$ Together with $\diam(M,g)=\pi$ the classical maximal diameter theorem of Cheng \cite{Cheng1975} yields that $(M^3,g)$ is isometric to $(\mathbb S^3,g_{\mathbb S^3})$. Conversely, the unit round sphere satisfies $\rho_{g_{\mathbb S^3}}\equiv2$, $\lambda_1(- \gamma\Delta_{g_{\mathbb S^3}}+\rho_{g_{\mathbb S^3}})=2$ and $\diam(\mathbb S^3)=\pi.$

\begin{rmk}\label{remark regularity}
    We justify here why we can assume that the function $u$ in the proof of Theorem \ref{Spectral Cheng} is smooth. Let $u \in C^{2,\alpha}(M)$ be a positive function such that $-\gamma \Delta_g u + \rho_g u \geq 2Ku$. First we claim that the upper bound $\frac{\pi}{\sqrt{K}}$ for the diameter holds. Take a sequence of positive smooth functions $(u_j)_j$ such that $u_j \rightarrow u$ in $C^2(M)$. Hence $-\gamma \Delta_g u_j + \rho_g u_j \rightarrow -\gamma \Delta_g u + \rho_g u$ uniformly. Thus there exists $\varepsilon_j \rightarrow 0$ such that $-\gamma \Delta_g u_j + \rho_g u_j \geq 2(K-\varepsilon_j)u_j$. For each $u_j$ we can apply the estimate already proved for smooth functions with $K_j \coloneqq K-\varepsilon_j$ to obtain \[\diam(M,g) \leq \dfrac{\pi}{\sqrt{K-\varepsilon_j}}.\] Taking the limit the claim follows. Now we prove the rigidity statement. Suppose $\diam(M,g) = \frac{\pi}{\sqrt{K}}$. Choose $\gamma <\beta <2$ and set 
    \[\lambda_\beta \coloneqq \sup \{C \colon \text{ there exists } w \in C^{2,\alpha}(M) \text{ such that } w>0 \text{ and } -\beta\Delta_gw + \rho_g w \geq C w \}.\] By hypothesis, for every smooth function with compact support $\varphi$ we have \[\int_M \gamma \vert \nabla \varphi \vert^2+\rho_g \varphi^2 \geq 2K \int_M \varphi^2. \] Since $\beta >\gamma,$ it follows that $\lambda_\beta \geq 2K$. Let $w>0$ be a $C^{2,\alpha}(M)$ function that satisfies $-\beta \Delta_g w + \rho_g w = \lambda_\beta w$ (for example, pick a positive first eigenfunction of the operator $-\beta \Delta_g + \rho_g$. The Hölder regularity is due to the fact that $\rho_g$ is in general only locally Lipschitz). By what we proved in the beginning of this remark, we have that: \[\frac{\pi}{\sqrt{K}} = \diam(M,g) \leq \pi \sqrt{\frac{2}{\lambda_\beta}} \leq \frac{\pi}{\sqrt{K}}\] and so $\lambda_\beta = 2K.$ Therefore  
    \begin{align*}
        2K \int_M w^2 = \int_M(\beta \vert \nabla w\vert^2 + \rho_g w^2) &=  \int_M(\gamma \vert \nabla w\vert^2 + \rho_g w^2) + (\beta-\gamma)\int_M \vert \nabla w \vert^2\\
        &\geq 2K \int_M w^2 + (\beta-\gamma)\int_M \vert \nabla w \vert^2.
    \end{align*}
    Since $\beta > \gamma$ we conclude that $w$ is constant. This implies that \[2K w = \lambda_\beta w = -\beta \Delta_g w + \rho_g w = \rho_g w\] and $\rho_g \equiv 2K$. The rigidity then follows. We can conclude that $u$ is constant as well. Indeed, since $-\gamma \Delta_g u +\rho_g u \geq 2Ku$ and $\rho_g \equiv 2K$, then $\Delta_g u \leq 0$. Since $M$ is compact, $u$ is a constant function.
\end{rmk}

\section{Sharpness of the range of $\gamma$} \label{section range gamma}

In this section we prove that Theorem \ref{Spectral Cheng} fails if we assume $\gamma = 2$, that is, the range $\gamma \in (0,2)$ is sharp.

\begin{prop}\label{sharpness for gamma}
There exists a one parameter family of smooth metrics $g_\varepsilon$ on $\mathbb S^3$, with $0<\varepsilon<1$, such that \[ \lambda_1(-2\Delta_{g_\varepsilon}+\rho_{g_\varepsilon})\geq2 \qquad \text{and}  \qquad \text{diam}(\mathbb S^3,g_\varepsilon)=\pi\]
but $(\mathbb S^3, g_\varepsilon)$ is not isometric to the round $3$-sphere.
\end{prop}

\begin{proof}
Fix $\varepsilon\in(0,1)$ and consider the rotationally symmetric metric \[ g_\varepsilon \coloneqq dt^2+a_\varepsilon(t)^2g_{\mathbb S^2}\] for $t \in [0,\pi]$, where $a_\varepsilon(t) \coloneqq \sin t\left(1-\varepsilon\sin^2t\right)$ and $g_{\mathbb S^2}$ denotes the round metric. Since $0<\varepsilon<1$, one has $a_\varepsilon(t)>0$ for $t\in(0,\pi)$ and moreover $a_\varepsilon(0)=a_\varepsilon(\pi)=0$, $a_\varepsilon'(0)=1$ and $a_\varepsilon'(\pi)=-1.$ Along with the fact that all even derivatives of $a_\varepsilon$ at the endpoints are zero, $g_\varepsilon$ extends smoothly to a Riemannian metric on $\mathbb S^3$. We first determine the smallest Ricci eigenvalue for these metrics. In order to do it, notice that in a general warped product metric \[ h=dt^2+a(t)^2g_{\mathbb S^2}\] the Ricci curvature behaves as follows: if $\partial t$ denotes the radial direction and $X$ is any unit vector tangent to the $\mathbb S^2$-factor,
\[\Ric(\partial_t,\partial_t) = -2\frac{a''}{a}, \qquad \Ric(X,X) = \frac{1-(a')^2-aa''}{a^2} \qquad\text{and} \qquad \Ric(X,\partial t) = 0.\] Therefore \[\Ric(X,X) - \Ric(\partial_t,\partial_t) = \frac{1-(a')^2+aa''}{a^2}.\]
In particular for $a=a_\varepsilon$, a direct computation gives
\[1-(a_\varepsilon')^2+a_\varepsilon a_\varepsilon'' = \varepsilon(4-3\varepsilon)\sin^4t.\]
Since $0<\varepsilon<1$, we have $\Ric(X,X) > \Ric(\partial_t,\partial_t)$ for every $t\in(0,\pi)$ and it follows that the smallest eigenvalue of the Ricci tensor is the radial one: \[\rho_{g_\varepsilon} = -2\frac{a_\varepsilon''}{a_\varepsilon}.\]
We now introduce the function \[u_\varepsilon(t) \coloneqq \frac{1}{1-\varepsilon\sin^2t}.\] This defines a smooth positive function on $\mathbb S^3$ which is non constant whenever $\varepsilon$ is not zero. For an arbitrary radial function $u=u(t)$, one has in the rotationally symmetric metric $h$: \[\Delta_h u= u'' + 2\frac{a'u'}{a} = \frac{(a^2u')'}{a^2}.\] Set $v\coloneqq au.$ We then have \[a^2u' =av'-a'v\] and hence \[\frac{\Delta_hu}{u} = \frac{v''}{v} - \frac{a''}{a}.\] Therefore we obtain the identity \[ -2\frac{\Delta_{g_\varepsilon}u_\varepsilon}{u_\varepsilon}+\rho_{g_\varepsilon} = -2\left(\frac{v_{\varepsilon}''}{v_\varepsilon} - \frac{a_{\varepsilon}''}{a_\varepsilon}\right) - 2\frac{a_\varepsilon''}{a_\varepsilon}= -2\frac{v_\varepsilon''}{v_\varepsilon}\] where the corresponding $v_\varepsilon(t)=\sin t$ by the choices of $a_\varepsilon$ and $u_\varepsilon$. This shows that \[-2\Delta_{g_\varepsilon}u_\varepsilon +\rho_{g_\varepsilon}u_\varepsilon= 2u_\varepsilon.\] Therefore $\lambda_1(-2\Delta_{g_\varepsilon}+\rho_{g_\varepsilon}) \geq2.$ It remains to compute the diameter of $(\mathbb S^3, g_\varepsilon)$. Note that as $a_\varepsilon(t) = \sin t\left(1-\varepsilon\sin^2t\right) \leq \sin t$, then the metrics satisfy \[ g_\varepsilon = dt^2+a_\varepsilon(t)^2g_{\mathbb S^2} \leq dt^2+\sin^2t\ g_{\mathbb S^2} = g_{\mathbb S^3}\] as quadratic forms, where $g_{\mathbb S^3}$ denotes the round metric in $\mathbb S^3$. It follows that \[d_{g_\varepsilon}(x,y) \leq d_{g_{\mathbb S^3}}(x,y) \leq \pi \] for every $x,y\in\mathbb S^3$. Hence $ \text{diam}(\mathbb S^3,g_\varepsilon)\leq \pi$. On the other hand, fixing $\theta \in S^2$, the curves $t \in [0,\pi] \mapsto (t,\theta)$ have length $\pi$ and every curve joining the north and south pole has length at least $\pi$. This gives the equality $ \text{diam}(\mathbb S^3,g_\varepsilon)=\pi $ as desired.
\end{proof}

\section{Discussion of the Case $n\neq3$} \label{section case n>3}

Recently Hong-Wang \cite{HongWang} obtained an explicit sharp uniform diameter estimate under spectral Ricci lower bounds in higher dimensions which refines the uniform estimate previously established by Antonelli-Xu \cite{AntonelliXuBG}. More precisely, let $(M^n,g)$ be a complete Riemannian manifold with dimension $n\geq 3$ and suppose that there exists a positive function $u$ such that \[-\gamma\Delta_g u+\rho_g u\geq (n-1)K u\] where $K>0$ and $0<\gamma<\frac{4}{n-1}.$ Theorem 1.3 in \cite{HongWang} says that under these conditions,
\begin{equation}\label{HW diameter}
    \diam(M,g) \leq 2\pi \sqrt{\frac{(n-1)-(n-2)\gamma}{(n-1)\bigl(4-(n-1)\gamma\bigr)K}} =: D_{n,\gamma,K}.
\end{equation}
Moreover, the constant in \eqref{HW diameter} is optimal among smooth closed examples. In particular, taking $K=1$ and defining $D_{n,\gamma} \coloneqq D_{n,\gamma,1}$ the result gives that the condition \[\lambda_1(-\gamma\Delta_g+\rho_g)\geq n-1\] implies $\diam(M,g)\leq D_{n,\gamma}.$ Notice that the case $n=3$ gives $D_{3,\gamma}=\pi$ for every $0<\gamma<2$. Due to our Theorem \ref{Spectral Cheng}, it makes natural in any dimension to ask whether the equality case admits a spectral analogue of Cheng's maximal diameter theorem. However the situation changes completely if $n>3$. A direct computation gives
\begin{equation}\label{D square}
    \frac{D_{n,\gamma}^2}{\pi^2}-1 = \frac{\gamma(n-3)^2}{(n-1)\bigl(4-(n-1)\gamma\bigr)}
\end{equation}
which shows that $D_{n,\gamma}>\pi$ whenever $n>3$ and $0<\gamma<\frac{4}{n-1}$. Since there are examples of closed $n$-manifolds with diameter almost attaining $D_{n,\gamma}$ (see Section 3.3 in \cite{HongWang}), in higher dimensions the natural spectral lower bound $\lambda_1(-\gamma\Delta_g+\rho_g)\geq n-1$ does not even imply the classical estimate $\diam(M)\leq\pi$. One could try to repair this by changing the spectral lower bound. Indeed, from \eqref{HW diameter} the choice
\begin{equation}\label{LambdaNGamma}
    \Lambda_{n,\gamma} \coloneqq \frac{4\bigl((n-1)-(n-2)\gamma\bigr)}{4-(n-1)\gamma}
\end{equation}
gives that the condition $\lambda_1(-\gamma\Delta_g+\rho_g) \geq\Lambda_{n,\gamma}$ ensures the bound $\diam(M,g)\leq\pi.$ However, \[\Lambda_{n,\gamma}-(n-1) = \frac{\gamma(n-3)^2}{4-(n-1)\gamma}.\] Hence for $n>3$ and $0<\gamma<\frac{4}{n-1}$, we have $\Lambda_{n,\gamma}>n-1.$ But the unit round sphere does not satisfy this strengthened spectral hypothesis because $\rho_{\mathbb S^n} \equiv n-1$ and any constant function $u_0$ satisfies $-\gamma \Delta_{\mathbb S^n}u_0 + \rho_{\mathbb S^n} u_0 = (n-1)u_0$. Therefore \eqref{LambdaNGamma} cannot provide a spectral Cheng theorem having the unit round sphere as the equality model. In fact, in dimensions $n>3$ no constant lower bound for $\lambda_1(-\gamma\Delta+\rho)$ yields a sharp diameter estimate  $\diam(M,g)\leq\pi$ with equality attained precisely by the unit round sphere. 

\begin{prop}\label{no spectral Cheng for higher dimensions}
Let $n>3$ and $0<\gamma<\frac{4}{n-1}$. There is no constant $C\in\R$ such that every closed Riemannian manifold $(M^n,g)$ satisfying $\lambda_1(-\gamma\Delta_g+\rho_g)\geq C$ also satisfies $\diam(M,g)\leq\pi$ with equality attained if and only if $(M^n,g)$ is isometric to the unit round sphere
$(\mathbb S^n,g_{\mathbb S^n})$.
\end{prop}

\begin{proof}
Suppose by contradiction that there exists such a constant $C$. Then for the round sphere $\mathbb S^n$, since $\rho_{g_{\mathbb S^n}}\equiv n-1$ and the constant functions are the first eigenfunctions, we must have $n-1\geq C.$ On the other hand, the optimality construction of Hong-Wang \cite[Proposition 3.4]{HongWang} provides a sequence of smooth rotationally symmetric metrics $g_j$ on $\mathbb S^n$, together with positive smooth functions $u_j$,
such that\[-\gamma \Delta_{g_j}u_j+\rho_{g_j}u_j  \geq (n-1)u_j\] and \[d_j\coloneqq\diam(S^n,g_j) \xrightarrow{\ j \rightarrow +\infty \ }{} D_{n,\gamma}.\] Since $D_{n,\gamma}>\pi$ by \eqref{D square}, for all sufficiently large $j$ one has $d_j>\pi.$ Set $c_j\coloneqq\frac{\pi}{d_j}<1$ and $\hat g_j\coloneqq c_j^2g_j$. Then $\diam(S^n,\hat g_j)=\pi.$ Under this constant rescaling of the metric, the operator $-\gamma\Delta+\rho$ scales by the inverse square factor and thus \[-\gamma\Delta_{\hat g_j} u_j + \rho_{\hat g_j} u_j = c_j^{-2}(-\gamma \Delta_{g_j} u_j +\rho_{g_j} u_j) \geq \left(\dfrac{d_j}{\pi}\right)^2(n-1)u_j\] We then obtain \[\lambda_1 \bigl(-\gamma\Delta_{\hat g_j} +\rho_{\hat g_j} \bigr)\geq \left(\dfrac{d_j}{\pi}\right)^2(n-1) > n-1\geq C.\] Thus $(\mathbb S^n,\hat g_j)$ satisfies the assumed spectral lower bound and has diameter exactly $\pi$, but it is not isometric to the unit round sphere. This contradicts the proposed rigidity statement.
\end{proof}

For the case $n=2$ the Ricci tensor of a Riemannian surface $(M^2,g)$ satisfies $\Ric_g = K_g \ g$ where $K_g$ is the Gaussian curvature of $M$. Hence $\rho_g = K_g$ and it is already known that the spectral condition $\lambda_1(-\gamma\Delta + K) \geq 1$ gives a diameter bound:

\begin{teo}[Theorem 1.4 in \cite{Xu2023}] \label{diam n=2}
    Given $0<\gamma < 4$, if a complete surface $(M^2,g)$ satisfies $\lambda_1(-\gamma\Delta_g + K_g) \geq 1$, then \[\diam(M,g) \leq \dfrac{2\pi}{\sqrt{(4-\gamma)}}.\] In particular $M$ must be a compact surface.
\end{teo}

Putting $D_\gamma \coloneqq \frac{2\pi}{\sqrt{4-\gamma}}$, since $D_\gamma>\pi$ this estimate does not recover the sharp diameter bound $\pi$ from Cheng's theorem and this suggests us that a spectral analogue of Cheng's rigidity theorem may fail in dimension two as well. In fact, we have the following proposition:

\begin{prop}
Let $0<\gamma<4$. Then there exists a smooth metric $g$ on $\mathbb S^2$, which is not isometric to the round metric, such that $\lambda_1(- \gamma\Delta_g+K_g)\geq 1$ and $\diam(\mathbb S^2,g)=\pi.$
\end{prop}

\begin{proof}
Fix $0<\varepsilon<\frac{\gamma}{\gamma+2}$ and consider on $\mathbb S^2$ the same rotationally symmetric metric considered in the proof of Proposition \ref{sharpness for gamma}, that is, \[g_\varepsilon \coloneqq dt^2+a_\varepsilon(t)^2\ d\theta^2, \qquad 0\leq t\leq\pi,\] where $a_\varepsilon(t) = \sin t\left(1-\varepsilon\sin^2t\right).$ We saw that this metric extends smoothly across the two poles and we now claim that \[ \lambda_1 \left(-\gamma\Delta_{g_\varepsilon}+K_{g_\varepsilon}\right)>1.\] To see this, define $u_\varepsilon(t) = \left(1-\varepsilon\sin^2t\right)^{-1/\gamma}$
which is a smooth positive function on $\mathbb S^2$. For a general rotationally symmetric metric $h=dt^2+a(t)^2\,d\theta^2$ and a function $u=u(t)$, one has the following relation: \[\Delta_h u = u''+\frac{a'}{a}u' \qquad \text{and} \qquad K_h = -\frac{a''}{a}.\] A direct computation for the particular case of $g_\varepsilon$ and $u_\varepsilon$ therefore gives \[\frac{-\gamma\Delta_{g_\varepsilon}u_\varepsilon + K_{g_\varepsilon}u_\varepsilon}{u_\varepsilon} = 1+ \frac{2\varepsilon\cos^2t}{\gamma\left(1-\varepsilon\sin^2t\right)^2} \left[\gamma-(\gamma+2)\varepsilon\sin^2t\right].\] Since $0<\varepsilon<\frac{\gamma}{\gamma+2}$ we have $(-\gamma\Delta_{g_\varepsilon}+K_{g_\varepsilon}) u_\varepsilon \geq u_\varepsilon$ with strict inequality away from the equator. Now let $\phi_1>0$ be the first eigenfunction of the operator $-\gamma\Delta_{g_\varepsilon}+K_{g_\varepsilon}$. Integrating by parts we obtain
\begin{align*}
    0 < \int_{\mathbb S^2} \phi_1 \left[(-\gamma\Delta_{g_\varepsilon}+K_{g_\varepsilon}) u_\varepsilon-u_\varepsilon \right] \ dA_{g_\varepsilon} &= \int_{\mathbb S^2} u_\varepsilon \left[(-\gamma\Delta_{g_\varepsilon}+K_{g_\varepsilon}) \phi_1-\phi_1 \right] \ dA_{g_\varepsilon}\\
    &= (\lambda_1-1) \int_{\mathbb S^2} u_\varepsilon\phi_1\ dA_{g_\varepsilon}.
\end{align*}
Since both $u_\varepsilon$ and $\phi_1$ are positive, it follows that $\lambda_1 \left(-\gamma\Delta_{g_\varepsilon}+K_{g_\varepsilon} \right)>1.$ The same argument in Proposition \ref{sharpness for gamma} shows that the diameter of $(\mathbb S^2, g_\varepsilon)$ is exactly $\pi$.
\end{proof}

Hu-Xu-Zhang proved in \cite{HuXuZhang2025} that the bound $D_\gamma$ for the diameter in Theorem \ref{diam n=2} is sharp in the following sense: there exists a sequence of complete Riemannian surfaces $(M_j,g_j)$ satisfying the spectral condition $\lambda_1(-\gamma\Delta_{g_j} + K_{g_j}) \geq 1$ and with $\lim\limits_{j\rightarrow +\infty} \diam(M_j,g_j) = D_\gamma$. The same proof of Proposition \ref{no spectral Cheng for higher dimensions} adapted for this sequence can be used to prove that in dimension two there is also no constant lower bound for $\lambda_1(-\gamma\Delta+K)$ yielding a sharp diameter estimate $\diam(M,g)\leq\pi$ with equality attained precisely by the unit round $2$-sphere (Proposition \ref{no spectral Cheng for n=2} below). In particular, the spectral analogue of Cheng's rigidity theorem fails in dimension two.

\begin{prop}\label{no spectral Cheng for n=2}
Let $0<\gamma< 4$. There is no constant $C\in\R$ such that every closed Riemannian surface $(M^2, g)$ satisfying $\lambda_1(-\gamma\Delta_g+K_g)\geq C$ also satisfies $\diam(M,g)\leq\pi$ with equality attained if and only if $(M^2,g)$ is isometric to the unit round $2$-sphere.
\end{prop}

\vspace{0.25cm}

\bibliographystyle{plain}
\bibliography{bibliography}
\end{document}